\documentclass[12pt]{amsart}
\usepackage[footskip=0.5in, headheight = 0.5in, top=1.25in, bottom=1.25in, right=1in, left=1in]{geometry}
\usepackage[utf8]{inputenc}
\usepackage[T1]{fontenc}
\usepackage{lmodern,microtype}
\usepackage[leqno]{amsmath}
\usepackage{enumerate,amssymb,xcolor,comment}
\usepackage[colorlinks,linkcolor=blue,citecolor=red]{hyperref}
\usepackage{tikz}
\usepackage{float}
\usepackage[center]{caption}

\newtheorem{question}{Question}[section]
\newtheorem{lemma}[question]{Lemma}
\newtheorem{theorem}[question]{Theorem}
\newtheorem{conjecture}[question]{Conjecture}
\newtheorem{corollary}[question]{Corollary}
\newtheorem{construction}[question]{Construction}

\newcommand\dd{\hbox{-}}
\newcommand\poi{\mathbb{N}}

\DeclareMathOperator{\tw}{tw}
\DeclareMathOperator{\dist}{dist}

\hypersetup{
  pdftitle={A simple layered-wheel-like construction},
  pdfauthor={Maria Chudnovsky, David Fischer, Sepehr Hajebi, Sophie Spirkl, Bartosz Walczak}
}

\title{A simple layered-wheel-like construction}

\author[M. Chudnovsky]{Maria Chudnovsky$^{\dag \ast}$}
\author[D. Fischer]{David Fischer$^{\dag}$}
\author[S. Hajebi]{Sepehr Hajebi$^{\ddag}$}
\author[S. Spirkl]{Sophie Spirkl$^{\ddag \mathsection}$}
\author[B. Walczak]{Bartosz Walczak$^{\parallel}$}

\thanks{$^{\dag}$ Princeton University, Princeton, NJ, USA}
\thanks{$^{\ddag}$ Department of Combinatorics and Optimization, University of Waterloo, Waterloo, Ontario, Canada}
\thanks{$^{\ast}$ Supported by NSF Grant DMS-2348219 and AFOSR grant FA9550-22-1-0083.}
\thanks{$^{\mathsection}$ We acknowledge the support of the Natural Sciences and Engineering Research Council of Canada (NSERC), [funding reference number RGPIN-2020-03912].
Cette recherche a \'et\'e financ\'ee par le Conseil de recherches en sciences naturelles et en g\'enie du Canada (CRSNG), [num\'ero de r\'ef\'erence RGPIN-2020-03912]. This project was funded in part by the Government of Ontario. This research was conducted while Spirkl was an Alfred P. Sloan Fellow.}
\thanks{$^{\parallel}$ Department of Theoretical Computer Science, Faculty of Mathematics and Computer Science, Jagiellonian University, Krak\'ow, Poland. Partially supported by the National Science Center of Poland grant 2019/34/E/ST6/00443.}

\allowdisplaybreaks

\begin{document}
\raggedbottom

\begin{abstract}
In recent years, there has been significant interest in characterizing the induced subgraph obstructions to bounded treewidth and pathwidth. While this has recently been resolved for pathwidth, the case of treewidth remains open, and prior work has reduced the problem to understanding the \emph{layered-wheel-like} obstructions -- graphs that contain large complete minor models with each branching set inducing a path, exclude large walls as induced minors, exclude large complete bipartite graphs as induced minors, and exclude large complete subgraphs.

There are various constructions of such graphs, but they are all rather involved. In this paper, we present a simple construction of layered-wheel-like graphs with arbitrarily large treewidth. Three notable features of our construction are: (a) the vertices of degree at least four can be made arbitrarily far apart; (b) the girth can be made arbitrarily large; and (c) every outerstring induced subgraph of the graphs from our construction has treewidth bounded by an absolute constant. In contrast, among several previously known constructions of layered wheels, none achieves (a); at most one satisfies either (b) or (c); and none satisfies both (b) and (c) simultaneously.

In particular, this is related to a former conjecture of Trotignon, that every graph with large enough treewidth, excluding large walls and large complete bipartite graphs as induced minors, and large complete subgraphs, must contain an outerstring induced subgraph of large treewidth. Our construction provides the first counterexample to this conjecture that can also be made to have arbitrarily large girth.
\end{abstract}

\maketitle

\section{Introduction}

\subsection{Background and the main result} \label{background}

Graphs in this paper have finite vertex sets, no loops and no parallel edges. Let $G=(V(G),E(G))$ be a graph. A graph $H$ is a \emph{minor of\/ $G$} if $H$ is isomorphic to a graph that can be obtained from $G$ by a series of vertex deletions, edge deletions, and edge contractions, and $H$ is an \emph{induced minor of\/ $G$} if $H$ is isomorphic to a graph that can be obtained from $G$ by a series of vertex deletions and edge contractions (and deleting the loops and the parallel edges produced in the contraction process). A \emph{tree decomposition} $(T, \chi)$ of $G$ consists of a tree $T$ and a map $\chi\colon V(T) \to 2^{V(G)}$ such that the following hold.
\begin{itemize}
    \item For every vertex $v \in V(G)$, there exists $t \in V(T)$ such that $v \in \chi(t)$.
    \item For every edge $v_1v_2 \in E(G)$, there exists $t \in V(T)$ such that $v_1, v_2 \in \chi(t)$.
    \item For every $v \in V(G)$, the subgraph of $T$ induced by $\{t \in V(T)\colon v \in \chi(t)\}$ is connected.
\end{itemize}

For a tree decomposition $(T, \chi)$ of $G$ with $V(T) = \{t_1, \hdots, t_n\}$, the sets $\chi(t_1), \hdots, \chi(t_n)$ are called the \emph{bags of\/ $(T, \chi)$}. The \emph{width of\/ $(T, \chi)$} is defined as $\max_{t \in V(T)} |\chi(t)|-1$. The \emph{treewidth} of $G$, denoted by $\tw(G)$, is the minimum width of a tree decomposition of $G$.

The Grid Theorem of Robertson and Seymour, Theorem~\ref{wallminor} below, fully describes the unavoidable subgraphs of graphs with large treewidth. For every $k\in \poi$, the \emph{$(k \times k)$-wall}, denoted by $W_{k \times k}$, is a planar graph with maximum degree three and with treewidth $k$ (see Figure \ref{fig:5x5wall}; a precise definition can be found in \cite{wallpaper}). Every subdivision of $W_{k \times k}$ is also a graph of treewidth $k$.

\begin{theorem}[Robertson and Seymour \cite{RS-GMV}]\label{wallminor}
There is a function\/ $f\colon \poi \rightarrow \poi$ such that every graph of treewidth at least\/ $f(k)$ contains a subdivision of\/ $W_{k \times k}$ as a subgraph.
\end{theorem}

\begin{figure}[t!]
\centering

\begin{tikzpicture}[scale=2,auto=left]
\tikzstyle{every node}=[inner sep=1.2pt, fill=black,circle,draw]
\centering

\node (s10) at (0,1.2) {};
\node(s12) at (0.6,1.2){};
\node(s14) at (1.2,1.2){};
\node(s16) at (1.8,1.2){};
\node(s18) at (2.4,1.2){};

\node (s20) at (0,0.9) {};
\node (s21) at (0.3,0.9) {};
\node(s22) at (0.6,0.9){};
\node (s23) at (0.9,0.9) {};
\node(s24) at (1.2,0.9){};
\node (s25) at (1.5,0.9) {};
\node(s26) at (1.8,0.9){};
\node (s27) at (2.1,0.9) {};
\node(s28) at (2.4,0.9){};
\node (s29) at (2.7,0.9) {};

\node (s30) at (0,0.6) {};
\node (s31) at (0.3,0.6) {};
\node(s32) at (0.6,0.6){};
\node (s33) at (0.9,0.6) {};
\node(s34) at (1.2,0.6){};
\node (s35) at (1.5,0.6) {};
\node(s36) at (1.8,0.6){};
\node (s37) at (2.1,0.6) {};
\node(s38) at (2.4,0.6){};
\node (s39) at (2.7,0.6) {};

\node (s40) at (0,0.3) {};
\node (s41) at (0.3,0.3) {};
\node(s42) at (0.6,0.3){};
\node (s43) at (0.9,0.3) {};
\node(s44) at (1.2,0.3){};
\node (s45) at (1.5,0.3) {};
\node(s46) at (1.8,0.3){};
\node (s47) at (2.1,0.3) {};
\node(s48) at (2.4,0.3) {};
\node (s49) at (2.7,0.3) {};

\node (s51) at (0.3,0.0) {};
\node (s53) at (0.9,0.0) {};
\node (s55) at (1.5,0.0) {};
\node (s57) at (2.1,0.0) {};
\node (s59) at (2.7,0.0) {};

\foreach \from/\to in {s10/s12, s12/s14,s14/s16,s16/s18}
\draw [line width=0.28mm] (\from) -- (\to);

\foreach \from/\to in {s20/s21, s21/s22, s22/s23, s23/s24, s24/s25, s25/s26,s26/s27,s27/s28,s28/s29}
\draw [line width=0.28mm] (\from) -- (\to);

\foreach \from/\to in {s30/s31, s31/s32, s32/s33, s33/s34, s34/s35, s35/s36,s36/s37,s37/s38,s38/s39}
\draw [line width=0.28mm] (\from) -- (\to);

\foreach \from/\to in {s40/s41, s41/s42, s42/s43, s43/s44, s44/s45, s45/s46,s46/s47,s47/s48,s48/s49}
\draw [line width=0.28mm] (\from) -- (\to);

\foreach \from/\to in {s51/s53, s53/s55,s55/s57,s57/s59}
\draw [line width=0.28mm] (\from) -- (\to);

\foreach \from/\to in {s10/s20, s30/s40}
\draw [line width=0.28mm] (\from) -- (\to);

\foreach \from/\to in {s21/s31,s41/s51}
\draw [line width=0.28mm] (\from) -- (\to);

\foreach \from/\to in {s12/s22, s32/s42}
\draw [line width=0.28mm] (\from) -- (\to);

\foreach \from/\to in {s23/s33,s43/s53}
\draw [line width=0.28mm] (\from) -- (\to);

\foreach \from/\to in {s14/s24, s34/s44}
\draw [line width=0.28mm] (\from) -- (\to);

\foreach \from/\to in {s25/s35,s45/s55}
\draw [line width=0.28mm] (\from) -- (\to);

\foreach \from/\to in {s16/s26,s36/s46}
\draw [line width=0.28mm] (\from) -- (\to);

\foreach \from/\to in {s27/s37,s47/s57}
\draw [line width=0.28mm] (\from) -- (\to);

\foreach \from/\to in {s18/s28,s38/s48}
\draw [line width=0.28mm] (\from) -- (\to);

\foreach \from/\to in {s29/s39,s49/s59}
\draw [line width=0.28mm] (\from) -- (\to);

\end{tikzpicture}

\caption{$W_{5 \times 5}$}
\label{fig:5x5wall}
\end{figure}

Theorem \ref{wallminor} also holds if ``subgraph'' is replaced by ``minor'' (in that case ``subdivision'' will not be necessary anymore). Recently, there has been growing interest in understanding the unavoidable \emph{induced} subgraphs of graphs with large treewidth. For instance, subdivided walls, complete graphs, and complete bipartite graphs are easily observed to have arbitrarily large treewidth. Line graphs of subdivided walls form another family of graphs with unbounded treewidth (recall that the \emph{line graph} $L(F)$ of a graph $F$ is the graph with vertex set $E(F)$, such that two vertices of $L(F)$ are adjacent if the corresponding edges of $F$ share an end). Since these four types of graphs do not contain each other as induced subgraphs, they must all be included in any potential family of induced subgraph obstructions to bounded treewidth; hence, they are often referred to as the ``basic obstructions''.

A full characterization of the induced subgraph obstructions to bounded treewidth remains unknown. Specifically, there are various constructions \cite{every-graph-essential, tw9, tw12, deathstar, complete-bipartite-minors, layeredwheel2, davies, pohoata2014unavoidable, layered-wheels} showing that forbidding the basic obstructions as induced subgraphs does not guarantee bounded treewidth. Furthermore, a recent result of Alecu, Bonnet, Bureo Villafana, and Trotignon \cite{every-graph-essential} suggests that an induced counterpart of the Grid Theorem is unlikely; nonetheless, there remains interest in understanding the obstructions to bounded treewidth from the induced standpoint. For example, we now know that every non-basic obstruction falls in one of only two categories: those that are ``complete bipartite induced minor models'' and those that are ``linear complete minor models''. Let us define these terms precisely. Given a graph $G$, a \emph{model in\/ $G$} is a graph $H$ with vertex set $\{X_1,\ldots,X_n\}$ such that the following hold.
\begin{enumerate}[{\rm (i)}]
    \item \label{model-connected} For all $i$ with $1 \leq i \leq n$, we have $X_i \subseteq V(G)$ and $G[X_i]$ is connected.
    \item \label{model-disjoint} For all $i$ and $j$ with $1 \leq i < j \leq n$, we have $X_i \cap X_j = \varnothing$.
    \item \label{model-edges} For all $i$ and $j$ with $1 \leq i < j \leq n$, if $X_i$ and $X_j$ are adjacent in $H$, then there exist $v_i \in X_i$ and $v_j \in X_j$ such that $v_iv_j\in E(G)$.
\end{enumerate}
We say that the model $H$ in $G$ is \emph{linear} if $G[X_i]$ is a path for all $i$ with $1\leq i\leq n$, and that $H$ is \emph{induced} if the converse to (\ref{model-edges}) is also true; that is, for all $i$ and $j$ with $1 \leq i < j \leq n$, if $X_i$ and $X_j$ are not adjacent in $H$, then $X_i$ and $X_j$ are anticomplete in $G$. Note that for any choice of the subsets $\{X_1,\ldots,X_n\}$ satisfying (\ref{model-connected}) and (\ref{model-disjoint}), there is exactly one induced model in $G$ with $\{X_1,\ldots,X_n\}$ as its vertex set, which we call the \emph{model in\/ $G$ induced by\/ $\{X_1,\ldots,X_n\}$}. It is easy to observe that a graph $F$ is a minor of $G$ if and only if $F$ is isomorphic to a model in $G$, and $F$ is an induced minor of $G$ if and only if $F$ is isomorphic to an induced model in $G$.

In \cite{complete-bipartite-minors}, three of us proved the following.

\begin{theorem}[Chudnovsky, Hajebi, Spirkl \cite{complete-bipartite-minors}]\label{thm:comp_model_general}
    For all\/ $r,s,t\in \poi$, there is a constant\/ $c=c(r,s,t)\in \poi$ such that for every graph\/ $G$ with\/ $\tw(G)>c$, one of the following holds.
    \begin{itemize}
        \item There is an induced subgraph of\/ $G$ isomorphic to one of\/ $K_{r+1}$, $K_{r,r}$, some subdivision of\/ $W_{r\times r}$ or the line graph of some subdivision of\/ $W_{r\times r}$.
        \item There is an induced model in\/ $G$ isomorphic to\/ $K_{s,s}$.
        \item There is a linear model in\/ $G$ isomorphic to\/ $K_t$.
    \end{itemize}
\end{theorem}

Moreover, the main result of \cite{complete-bipartite-minors} characterizes the unavoidable induced subgraphs of graphs containing a large complete bipartite induced minor. Therefore, in order to establish a full analog of Theorem~\ref{wallminor} for induced subgraphs, it remains to understand those obstructions that are linear complete models which do not contain basic obstructions of large treewidth as induced subgraphs, and do not contain large complete bipartite graphs as induced minors.

We call these obstructions \emph{layered-wheel-like}, and the naming is explained by the fact that the only known examples of such obstructions are (variations of) the so-called \emph{layered wheels}. Indeed, there are several constructions of layered wheels, by Sintiari and Trotignon \cite{layered-wheels}, by Chudnovsky and Trotignon \cite{layeredwheel2} and by Alecu, Bonnet, Villafana and Trotignon \cite{every-graph-essential}. But they are all quite intricate. In this paper, we present a new construction of layered-wheel-like graphs with arbitrarily large treewidth. Our construction is substantially simpler than all previous ones and simultaneously achieves several key properties, no two of which are attained individually by any earlier construction. Explicitly, our main result is as follows.

\begin{theorem}\label{main}
    There exist\/ $L,t_0\in \poi$ such that for all\/ $g,k\in \poi$, there is a graph\/ $G_k^g$ with the following properties.
    \begin{enumerate}[{\rm (i)}]
        \item \label{main-1} There is a linear model in\/ $G^g_k$ isomorphic to\/ $K_k$. In particular, $\tw(G_k^g)\geq k-1$.
        \item \label{main-2} $G^g_k$ does not contain\/ $W_{t_0 \times t_0}$ or\/ $K_{t_0,t_0}$ as an induced minor.
        \item \label{main-3} If\/ $u,v\in V(G^g_k)$ both have degree at least four in\/ $G^g_k$ and\/ $u \neq v$, then\/ $\dist_{G^g_k}(u,v) \geq 2^g$.
        \item \label{main-4} $G^g_k$ has girth at least\/ $g$.
        \item \label{main-5} If\/ $H$ is an induced subgraph of\/ $G^g_k$ and\/ $H$ is an outerstring graph, then\/ $\tw(H) \leq L$.
    \end{enumerate}
\end{theorem}

No preexisting construction of layered wheels satisfies (\ref{main-3}) (and in fact they all contain every tree as induced subgraphs). Moreover, only one construction, namely the ``theta-free'' one \cite{layered-wheels}, allows for arbitrarily large girth, and the most recent construction of layered wheels \cite{every-graph-essential} (which is also the most complicated one) is the only one that achieves (\ref{main-5}) -- in particular, no other construction satisfies (\ref{main-4}) and (\ref{main-5}) at once.

Property (\ref{main-5}) is closely related to the following conjecture of Trotignon \cite{conjecture}. A \emph{string representation} of a graph $G$ is an assignment of the vertices of $G$ to curves in the plane such that the curves corresponding to two vertices $u,v$ intersect if and only if $uv \in E(G)$. An \emph{outerstring representation} of $G$ is a string representation of $G$ in which all curves lie in the upper half of the plane and each curve has exactly one endpoint on the $x$-axis. A graph $G$ is a \emph{string graph} if it has a string representation, and it is an \emph{outerstring graph} if it has an outerstring representation.

\begin{conjecture}[Trotignon \cite{conjecture}] \label{original-conjecture}
    For all\/ $r,t\in \poi$, there exists\/ $c = c(r,t)\in \poi$ such that for every graph\/ $G$, if\/ $G$ does not contain\/ $W_{t \times t}$ or\/ $K_{t,t}$ as an induced minor, and every induced subgraph\/ $H$ of\/ $G$ that is an outerstring graph satisfies $\tw(H) \leq r$, then\/ $\tw(G) \leq c$.
\end{conjecture}

This conjecture is already refuted by the construction from \cite{every-graph-essential}. Our main construction provides a different counterexample to Conjecture \ref{original-conjecture}, which has the advantage of being much simpler in structure, having arbitrarily large girth, and not containing all trees as induced subgraphs. In fact, our proof gives a slightly stronger property than (\ref{main-5}); see Theorem~\ref{stronger}.

\subsection{Definitions and notation} \label{defs}

We write $\poi$ for the set of positive integers. Let $G$ be a graph. An \emph{induced subgraph of\/ $G$} is a graph $H$ obtained from $G$ by deleting vertices. For $X \subseteq V(G)$, we let $G[X]$ denote the induced subgraph of $G$ with vertex set $X$, and $G \setminus X$ denotes $G[V(G) \setminus X]$. We often use $X$ to denote both the set $X$ of vertices and the induced subgraph $G[X]$. We say that two subsets $X,Y \subseteq V(G)$ are \emph{anticomplete} if $X \cap Y = \varnothing$ and there is no edge of $G$ with an end in $X$ and an end in $Y$. For $v \in V(G)$, we let $N_G(v)$ denote the set of all vertices in $G$ adjacent to $v$, and we write $N_G[v]=N_G(v) \cup \{v\}$. The \emph{degree of\/ $v$ in\/ $G$}, denoted by $\deg_G(v)$, is $|N_G(v)|$. The \emph{maximum degree of\/ $G$} is $\max_{v \in V(G)} \deg_G(v)$, and the \emph{minimum degree of\/ $G$} is $\min_{v \in V(G)} \deg_G(v)$. For $X \subseteq V(G)$, we let $N_G(X)$ denote the set of all vertices in $G\setminus X$ with at least one neighbor in $X$, and we define $N_G[X]=N_G(X) \cup X$.

A \emph{path} $P$ is a graph with vertex set $\{v_1,\ldots,v_n\}$ and edge set $\{v_iv_{i+1}\colon 1 \leq i \leq n-1 \}$. Such a path is denoted by $v_1 \dd v_2 \dd \cdots \dd v_{n-1} \dd v_n$. We say that $v_1$ and $v_n$ are the \emph{ends} of $P$, and that $P$ is a \emph{path from\/ $v_1$ to\/ $v_n$}. The \emph{interior of\/ $P$}, denoted by $P^*$, is $P \setminus \{v_1,v_n\}$. The \emph{length} of $P$ is given by $|E(P)| = n - 1$. A \emph{path in\/ $G$ from\/ $u$ to\/ $v$} is an induced subgraph of $G$ that is a path from $u$ to $v$. If $G$ is connected, the \emph{distance from\/ $u$ to\/ $v$ in\/ $G$}, denoted by $\dist_G(u,v)$, is the length of the shortest path from $u$ to $v$ in $G$. For $u,v \in P$, the \emph{subpath of\/ $P$ from\/ $u$ to\/ $v$} is the (unique) path in $P$ with ends $u$ and $v$. Two paths $P_1$ and $P_2$ in $G$ are \emph{internally anticomplete} if $P_1^*$ and $P_2^*$ are anticomplete.

The $k$-vertex cycle $C_k$ is obtained from a path on $k$ vertices by adding an edge between its two ends. The \emph{girth} of a graph $G$ is the smallest $k$ such that $G$ contains $C_k$ as an (induced) subgraph (and is infinite if $G$ does not contain any cycles). A \emph{hole in\/ $G$} is an induced subgraph of $G$ isomorphic to $C_k$ for some $k \geq 4$.

\section{The main construction} \label{sec:construction}

In this section, we give the description of the graph family that proves Theorem~\ref{main}.

\begin{construction}
    \label{themainconstruction}
    For\/ $g,k \in \poi$, let\/ $G_k^g$ be the graph constructed as follows.
    \begin{enumerate}[{\rm (i)}]
        \item $V(G_k^g)$ is partitioned into\/ $k$ paths\/ $P_1,\ldots,P_k$ of length\/ $2^{k+g}$, where\/ $P_i=P_i^0\dd \cdots\dd P_i^{2^{k+g}}$ for each\/ $i$ with\/ $1 \leq i \leq k$.
        \item For all\/ $i$ and\/ $j$ with\/ $1 \leq i < j \leq k$, the vertices\/ $P_i^x$ and\/ $P_j^y$ are adjacent in\/ $G^g_k$ if\/ $x = y = b \cdot 2^{k-i+g}$ for some odd integer\/ $b$; namely, $b \in \{1,3,\ldots,2^i - 1\}$.
    \end{enumerate}
\end{construction}

See Figure \ref{fig:G34}. We say that the vertex $P_\ell^x \in V(G_k^g)$ has \emph{layer\/ $\ell$} and \emph{index\/ $x$}. For $1 \leq \ell \leq k$ and $0 \leq x_1,x_2 \leq 2^{k+g}$, we let $P_\ell[x_1 : x_2]$ denote the subpath of $P_\ell$ with ends $P_\ell^{x_1}$ and $P_\ell^{x_2}$. For the purposes of analysis, we partition $V(G_k^g)$ into three sets: \emph{big}, \emph{medium}, and \emph{small vertices}; where a vertex $P_i^x \in V(G_k^g)$ is \emph{big} if $x = b \cdot 2^{k - i + g}$ for some odd $b$, \emph{medium} if $x = b \cdot 2^{k - j + g}$ for some odd $b$ and some $1 \leq j < i$, and \emph{small} otherwise. The sets of big, medium, and small vertices of $G_k^g$ are denoted by $B(G_k^g)$, $M(G_k^g)$, and $S(G_k^g)$, respectively. We note that all big vertices in $V(G_k^g) \setminus V(P_k)$ have degree at least $3$, all medium vertices have degree exactly $3$, and all small vertices have degree at most $2$.

\begin{figure}[t!]
\centering

\begin{tikzpicture}[scale=1,auto=left, line width=0.28mm]
\tikzstyle{every node}=[inner sep=1.2pt, fill=black,circle,draw]
\centering

    \draw (1, 1) node[draw=none,fill=none,left] {$P_5^0$} -- (12, 1) node[draw=none,fill=none,right] {$P_5^{64}$};
    \draw (1, 2) node[draw=none,fill=none,left] {$P_4^0$} -- (12, 2) node[draw=none,fill=none,right] {$P_4^{64}$};
    \draw (1, 3) node[draw=none,fill=none,left] {$P_3^0$} -- (12, 3) node[draw=none,fill=none,right] {$P_3^{64}$};
    \draw (1, 4) node[draw=none,fill=none,left] {$P_2^0$} -- (12, 4) node[draw=none,fill=none,right] {$P_2^{64}$};
    \draw (1, 5) node[draw=none,fill=none,left] {$P_1^0$} -- (12, 5) node[draw=none,fill=none,right] {$P_1^{64}$};

    \foreach \i in {1,...,5} {
        \node() at (6.5, \i) {};
        \node() at (1,\i) {};
        \node at (12,\i) {};
    }
    \draw (6.5, 5) edge[line width=0.28mm, line width=0.28mm, out = -90, in = 90] (6.5, 4);
    \draw (6.5, 5) edge[line width=0.28mm, line width=0.28mm, out = -105, in = 105] (6.5, 3);
    \draw (6.5, 5) edge[line width=0.28mm, out = -112, in = 112] (6.5, 2);
    \draw (6.5, 5) edge[line width=0.28mm, out = -120, in = 120] (6.5, 1);

    \foreach \i in {1,...,4} {
        \node() at (3.75,\i) {};
        \node() at (9.25,\i) {};
    }
    \draw (3.75, 4) edge[line width=0.28mm, out = -90, in = 90] (3.75, 3);
    \draw (3.75, 4) edge[line width=0.28mm, out = -105, in = 105] (3.75, 2);
    \draw (3.75, 4) edge[line width=0.28mm, out = -112, in = 112] (3.75, 1);
    \draw (9.25, 4) edge[line width=0.28mm, out = -90, in = 90] (9.25, 3);
    \draw (9.25, 4) edge[line width=0.28mm, out = -105, in = 105] (9.25, 2);
    \draw (9.25, 4) edge[line width=0.28mm, out = -112, in = 112] (9.25, 1);

    \foreach \i in {1,...,3} {
        \node() at (2.375, \i) {};
        \node() at (5.125,\i) {};
        \node() at (7.875,\i) {};
        \node() at (10.625,\i) {};
    }
    \draw (2.375, 3) edge[line width=0.28mm, out = -90, in = 90] (2.375, 2);
    \draw (2.375, 3) edge[line width=0.28mm, out = -112, in = 112] (2.375, 1);
    \draw (5.125, 3) edge[line width=0.28mm, out = -90, in = 90] (5.125, 2);
    \draw (5.125, 3) edge[line width=0.28mm, out = -112, in = 112] (5.125, 1);
    \draw (7.875, 3) edge[line width=0.28mm, out = -90, in = 90] (7.875, 2);
    \draw (7.875, 3) edge[line width=0.28mm, out = -112, in = 112] (7.875, 1);
    \draw (10.625, 3) edge[line width=0.28mm, out = -90, in = 90] (10.625, 2);
    \draw (10.625, 3) edge[line width=0.28mm, out = -112, in = 112] (10.625, 1);

    \foreach \i in {1,2} {
        \node() at (1.6875,\i) {};
        \node() at (3.0625,\i) {};
        \node() at (4.4375,\i) {};
        \node() at (5.8125,\i) {};
        \node() at (7.1875,\i) {};
        \node() at (8.5625,\i) {};
        \node() at (9.9375,\i) {};
        \node() at (11.3125,\i) {};
    }

    \draw(1.6875,2) edge[line width=0.28mm, out = -90, in = 90] (1.6875,1);
    \draw(3.0625,2) edge[line width=0.28mm, out = -90, in = 90] (3.0625,1);
    \draw(4.4375,2) edge[line width=0.28mm, out = -90, in = 90] (4.4375,1);
    \draw(5.8125,2) edge[line width=0.28mm, out = -90, in = 90] (5.8125,1);
    \draw(7.1875,2) edge[line width=0.28mm, out = -90, in = 90] (7.1875,1);
    \draw(8.5625,2) edge[line width=0.28mm, out = -90, in = 90] (8.5625,1);
    \draw(9.9375,2) edge[line width=0.28mm, out = -90, in = 90] (9.9375,1);
    \draw(11.3125,2) edge[line width=0.28mm, out = -90, in = 90] (11.3125,1);

\end{tikzpicture}

\caption{$G_5^1$ (internal small vertices are not drawn)}
\label{fig:G34}
\end{figure}

Some basic properties of the graph $G^g_k$ are listed below (the proofs are easy and we leave the details to the reader to check).

\begin{lemma}\label{props}
    For all\/ $g,k\in \poi$, the graph\/ $G_k^g$ has the following properties.
    \begin{enumerate}[{\rm (i)}]
        \item \label{props-6} There is a linear model in\/ $G^g_k$ isomorphic to\/ $K_k$. In particular, $\tw(G_k^g)\geq k-1$.
        \item \label{props-1} $G_k^g$ is triangle-free.
        \item \label{props-5} $G_k^g$ has girth at least\/ $g$.
        \item \label{props-2} If\/ $u,v \in V(G_k^g)$ are adjacent, then\/ $\deg_{G_k^g}(u) \leq 3$ or\/ $\deg_{G_k^g}(v) \leq 3$, and\/ $\{u,v\} \nsubseteq B(G_k^g)$.
        \item \label{props-3} If\/ $u,v \in B(G_k^g)$ with\/ $u \neq v$, then\/ $\dist_{G_k^g}(u,v) \geq 2^g$. In particular, if\/ $u, v \in V(G_k^g)$ with\/ $u \neq v$ are non-adjacent and both have degree at least\/ $4$, then\/ $\dist_{G_k^g}(u,v) \geq 2^g$.
        \item \label{props-4} If\/ $u \in V(G_k^g)$ with\/ $\deg_{G_k^g}(u) = 3$, then there is at most one\/ $v \in N_{G_k^g}(u)$ such that\/ $\deg_{G_k^g}(v) \geq 3$.
    \end{enumerate}
\end{lemma}

% \begin{proof}
%     To see that (\ref{props-1}) holds, note that for $1 \leq i \leq k$, no two adjacent vertices $u,v \in V(P_i)$ share a neighbor. If $u = P_i^x$, $v = P_j^y$, $w = P_\ell^z$ for $0 \leq i < j < \ell \leq k$, then $uv,uw \in E(G_k^g)$ only if $x = y = z = b \cdot 2^{k-i+g}$ for some odd $b$, and then it follows from the construction that $v$ and $w$ are non-adjacent. Statements (\ref{props-2}), (\ref{props-3}), (\ref{props-4}), and (\ref{props-5}) follow similarly from the construction of $G_k^g$. Statement (\ref{props-6}) holds because $G_k^g$ contains $K_k$ as a minor, which can be seen from the model induced by $\{P_1,\ldots,P_k\}$. Since $\tw(K_k) = k - 1$ and treewidth is non-increasing under taking minors, the result follows.
% \end{proof}

Therefore, in order to prove Theorem~\ref{main}, it remains to show that for all $g,k\in \poi$, there is no large wall in $G^g_k$ as an induced minor, there is no large complete bipartite graph in $G^g_k$ as an induced minor, and all outerstring induced subgraphs of $G^g_k$ have small treewidth. We will prove the latter in the next three sections, and then we will complete the proof of Theorem \ref{main} in the final section.

\section{No large wall as an induced minor} \label{no-large-wall}

The main result of this section is the following lemma.

\begin{lemma} \label{no-wall-im}
    There exists\/ $h_0\in \poi$ such that for all\/ $g,k\in \poi$, the graph\/ $G_k^g$ does not contain\/ $W_{h_0 \times h_0}$ as an induced minor.
\end{lemma}

The proof involves series-parallel graphs, which are exactly the graphs with treewidth at most $2$. The definition given here is adapted from \cite{series-parallel}. A \emph{two-terminal graph} is a graph $G$ where two distinct vertices $s,t \in V(G)$ are designated as the \emph{terminals} of $G$, where $s$ is the \emph{source} and $t$ is the \emph{sink}. If $G$ is a two-terminal graph with source $s$ and sink $t$, and $H$ is a two-terminal graph with source $s'$ and sink $t'$, then the \emph{series-composition of\/ $G$ and\/ $H$} is the two-terminal graph obtained by combining $G$ and $H$ via identifying $t$ and $s'$ into a single vertex (that is, replacing them by a single vertex with neighbor set $N_G(t) \cup N_H(s')$), and declaring $s$ as the source and $t'$ as the sink of the resulting graph. The \emph{parallel-composition of\/ $G$ and\/ $H$} is the two-terminal graph obtained by combining $G$ and $H$ via identifying $s$ and $s'$ into a single vertex and declaring that vertex as the source of the resulting graph, and identifying $t$ and $t'$ into a single vertex and declaring that vertex as the sink of the resulting graph. A two-terminal graph $G$ with source $s$ and sink $t$ is \emph{$(s,t)$-series-parallel} if $G$ can be obtained from copies of $K_2$ via a sequence of series-compositions and parallel-compositions, where each copy of $K_2$ begins with one of its vertices as the source and the other as the sink. A graph $G$ is \emph{series-parallel} if it is a subgraph of some $(s,t)$-series-parallel graph.

We now return to the proof of Lemma \ref{no-wall-im}. It is a consequence of Lemma 3.6 in \cite{stone-wall} that, for every constant $t > 0$, there is a constant $h_t > 0$ such that every graph containing $W_{h_t \times h_t}$ as an induced minor contains a subdivision of $W_{t \times t}$ or the line graph of a subdivision of $W_{t \times t}$ as an induced subgraph. Recall that for all $g,k\in \poi$, the graph $G_k^g$ is triangle-free, and so it does not contain the line graph of a subdivision of any wall. Thus, to show the existence of the desired $h_0$, it suffices to show that $G_k^g$ does not contain a subdivision of a large wall as an induced subgraph. This will be accomplished in Lemma~\ref{no9wall}, but we first need two more lemmas. Our goal is to show that for every induced subgraph of $G_k^g$, if we contract all edges not contained in $P_1, \dots, P_k$ (the vertical edges in Figure \ref{fig:G34}), then the resulting graph is series-parallel.

\begin{lemma} \label{general-series-parallel}
    Let\/ $F$ be a two-terminal graph with source\/ $s$ and sink\/ $t$, and the following specifications.
    \begin{enumerate}[{\rm (i)}]
        \item There are paths\/ $P_1,\ldots,P_k$ from\/ $s$ to\/ $t$ such that\/ $V(F) = V(P_1) \cup \ldots \cup V(P_k)$ and\/ $V(P_i^*) \cap V(P_j^*) = \varnothing$ for\/ $i \neq j$.
        \item For\/ $1 \leq i \leq k$ and every\/ $v \in P_i^*$, exactly one of the following holds.
        \begin{itemize}
            \item $N_F(v) \setminus V(P_i) = \varnothing$; we let\/ $S'(F)$ denote the set of all vertices in\/ $F$ for which this outcome holds.
            \item $i < k$, the vertex\/ $v$ has exactly one neighbor in\/ $P_j^*$ for every\/ $j$ with\/ $i < j \leq k$, and\/ $v$ has no other neighbors apart from its neighbors in\/ $P_i$; in this case, we write\/ $N_F'[v] = (N_F(v) \setminus V(P_i)) \cup \{v\}$. We let\/ $B'(F)$ denote the set of all vertices in\/ $F$ for which this outcome holds.
            \item $i > 1$, and\/ $v$ has exactly one neighbor in\/ $P_j^*$ for some\/ $j < i$, and\/ $v$ has no other neighbors outside\/ $P_i$.
        \end{itemize}
        \item Let\/ $v_1$ and\/ $v_2$ be distinct vertices in\/ $B'(F) \cap V(P_i^*)$ for some\/ $i$ with\/ $1 \leq i < k$ such that\/ $\dist_{P_i}(s,v_1) < \dist_{P_i}(s,v_2)$. Then, for every\/ $j$ with\/ $i < j \leq k$, we have\/ $\dist_{P_j}(s,u_1) < \dist_{P_j}(s,u_2)$, where\/ $u_1$ is the unique neighbor of\/ $v_1$ in\/ $P_j^*$ and\/ $u_2$ is the unique neighbor of\/ $v_2$ in\/ $P_j^*$.
    \end{enumerate}
    Then the model in\/ $F$ induced by the set
    \[ \{ N_F'[v]\colon v \in B'(F) \} \cup \{\{w\}\colon w \in S'(F)\} \cup \{\{s\},\{t\}\}, \]
    denoted by\/ $c(F)$, is an\/ $(s,t)$-series-parallel graph.
\end{lemma}

\begin{proof}
    We proceed by induction on $|B'(F)|$. If $|B'(F)| = 0$, then $c(F)$ is isomorphic to $F$ and consists of $k$ pairwise internally anticomplete paths from $s$ to $t$, and so it is $(s,t)$-series-parallel. We now assume that $|B'(F)| > 0$. We may further assume that $P_1$ contains a vertex in $B'(F)$; if this is not the case, then we may apply the argument on $F[V(P_r) \cup \ldots \cup V(P_k)]$, where $r$ is minimal such that $P_r$ contains a vertex in $B'(F)$, and then take the parallel-composition of this graph with $P_1,\ldots,P_{r-1}$ to show that $F$ is $(s,t)$-series-parallel. Now, let $b \in B'(F) \cap V(P_1^*)$ with $\dist_{P_1}(s,b)$ minimal. Since $b$ has a neighbor in each of $P^*_2,\ldots,P^*_k$ in $F$, it follows that $F \setminus N'[b]$ has two components, one containing $s$ and the other containing $t$. Let $C_s$ and $C_t$, respectively, denote these components.

    We now observe that the model $F_s$ in $F$ induced by $\{ \{v\}\colon v \in V(C_s) \} \cup \{N'[b]\}$ satisfies the requirements of the lemma, with terminals $s$ and $N'[b]$ and the paths from $s$ to $N'[b]$ being the truncated versions of the paths $P_1,\ldots,P_k$. Clearly $B'(F_s) \subseteq B'(F)$; furthermore, this containment is strict, since $b \in B'(F)$, and $b \notin B'(F_s)$ as it is part of the terminal vertex $N_F'[b]$ in $F_s$. By induction, we find that $c(F_s)$ is $(s,N'_F[b])$-series-parallel. Similarly, $c(F_t)$ is $(N_F'[b],t)$-series-parallel, where $F_t$ is the model in $F$ induced by $\{ \{v\}\colon v \in V(C_t) \} \cup \{N_F'[b]\}$. Since $c(F)$ is the series-composition of $c(F_s)$ and $c(F_t)$, it follows that $c(F)$ is series-parallel.
\end{proof}

\begin{lemma} \label{H'-SP}
    Let\/ $g,k\in \poi$, and let\/ $H$ be an induced subgraph of\/ $G_k^g$. For every vertex\/ $b \in B(G_k^g) \cap V(H)$, let\/ $N_H^M(b) = N_H(b) \cap M(G_k^g)$ and\/ $N_H^M[b] = N_H^M(b) \cup \{b\}$. Let $B' = \{ N_H^M[b]\colon b \in B(G_k^g) \cap V(H) \}$, and let\/ $H'$ be the model in\/ $H$ induced by\/ $B' \cup \{\{v\}\colon v \in V(H),\: v \notin \bigcup_{Y \in B'} Y\}$; equivalently, $H'$ is obtained from\/ $H$ by contracting every edge\/ $e \in E(H)$ that has both ends in\/ $N_H^M[b]$ for some\/ $b \in B(G_k^g) \cap V(H)$. Then\/ $H'$ is series-parallel (in particular, $\tw(H') \leq 2$).
\end{lemma}

\begin{proof}
    Let $X \subseteq V(G_k^g)$ be such that $H = G_k^g \setminus X$. Put $X_B = X \cap B(G_k^g)$, $X_M = X \cap M(G_k^g)$, and $X_S = X \cap S(G_k^g)$. Let $F$ be obtained from $G_k^g$ by removing, for each $b \in X_B$, every edge incident with $b$ whose other end is in $M(G_k^g)$. Next, we add auxiliary terminal vertices $T^-$ and $T^+$ to $F$, where $T^-$ is incident to $P_\ell^0$ for $1 \leq \ell \leq k$ and $T^+$ is incident to $P_\ell^{2^{k+g}}$ for $1 \leq \ell \leq k$. Applying Lemma \ref{general-series-parallel} gives that $c(F)$ is $(T^-,T^+)$-series-parallel, where $c(F)$ is as defined there. We now show that $H'$ is a subgraph of $c(F)$. We observe that $H'$ can be obtained from $c(F)$ as follows.
    \begin{enumerate}[{\rm (i)}]
        \item Delete $T^-$ and $T^+$.
        \item For $v \in X_S$, delete the vertex $\{v\}$ from $c(F)$.
        \item For $m \in X_M$, let $b$ be the unique neighbor (in $G_k^g$) of $m$ in $B(G_k^g)$, and let $s_1$ and $s_2$ be the two neighbors of $m$ in $S(G_k^g)$.
        \begin{itemize}
            \item If $b \in X_B$, delete $\{m\}$ from $c(F)$ (since $b \in X_B$, there is no edge between $b$ and $m$ in $F$, and so $\{ m \}$ is a vertex of $c(F)$).
            \item If $b \notin X_B$, there is a vertex $S$ in $c(F)$ such that $b,m \in S$. If $s_1 \notin X_S$, delete the edge between $S$ and $\{s_1\}$ in $c(F)$. Similarly, if $s_2 \notin X_S$, delete the edge between $S$ and $\{s_2\}$ in $c(F)$.
        \end{itemize}
        \item For $b \in X_B$, delete $\{b\}$ from $c(F)$.
    \end{enumerate}
    Thus $H'$ is a subgraph of the series-parallel graph $c(F)$, and so $H'$ is series-parallel.
\end{proof}

\begin{lemma}\label{no9wall}
    For all\/ $g,k\in \poi$, the graph\/ $G_k^g$ has no induced subgraph isomorphic to a subdivision of\/ $W_{5 \times 5}$.
\end{lemma}

\begin{proof}
Suppose that $J$ is an induced subgraph of $G_k^g$ that is isomorphic to a subdivision of $W_{5 \times 5}$. Say a vertex $v \in V(J)$ is a \emph{branch vertex} of $J$ if $v$ is also a vertex of $W_{5 \times 5}$; that is, $v$ is not a vertex that was created in the subdivision process. Let $H$ be a subdivision of $W_{3 \times 3}$ in $J$ such that every two branch vertices of $H$ have distance at least $3$ in $H$. Let $H'$ be as defined in Lemma \ref{H'-SP}, and for each vertex $u$ of $H$, let $p(u)$ be the unique vertex of $H'$ such that $u \in p(u)$. Observe that, for vertices $u,v \in V(H)$, $p(u) = p(v)$ can only hold if $u$ and $v$ have distance at most two in $H$; in particular, for every pair of distinct branch vertices $u,v \in V(H)$ we have $p(u) \neq p(v)$. This implies that $H'$ contains a subdivision of $W_{3 \times 3}$ (whose branch vertices are exactly the branch vertices of $H$). But now $\tw(H')>2$, which contradicts Lemma~\ref{H'-SP}.
\end{proof}

\section{No large complete bipartite induced minor}

In this section we prove the following result.

\begin{lemma} \label{no-bipartite-im}
    There exists\/ $r_0\in \poi$ such that for all\/ $g,k\in \poi$, the graph\/ $G_k^g$ does not contain\/ $K_{r_0,r_0}$ as an induced minor.
\end{lemma}

We need to prepare for the proof. A graph $T$ is a \emph{wide theta of width\/ $m$} if for some distinct vertices $a,b \in V(T)$ (called the \emph{ends of\/ $T$}), there are $m$ pairwise internally anticomplete paths $P_1,\ldots,P_m$ from $a$ to $b$ in $T$, each of length at least $2$, and $T$ has no other vertices or edges. The following is an immediate corollary of 1.3 in \cite{complete-bipartite-minors}.

\begin{lemma}[Chudnovsky, Hajebi, Spirkl \cite{complete-bipartite-minors}]\label{complete-induced-minor}
    For all\/ $h\in \poi$, there exists\/ $r = r(h)\in \poi$ such that if\/ $G$ does not contain a wide theta of width\/ $8$ as an induced subgraph or\/ $W_{h \times h}$ as an induced minor, then\/ $G$ does not contain\/ $K_{r,r}$ as an induced minor.
\end{lemma}

Accordingly, in what follows, we will show that for all $g,k\in \poi$ and every pair of distinct vertices $u,v \in G_k^g$, there can be at most seven pairwise internally anticomplete paths between $u$ and $v$ in $G_k^g$. This will be achieved in Corollary \ref{leq-seven-paths}, which, combined with Lemmas \ref{no-wall-im} and \ref{complete-induced-minor}, gives a proof of Lemma \ref{no-bipartite-im}. We remark that through further casework the proof shown here can be extended to obtain a bound of at most three paths, which is tight.

Note that for all $g,k\in \poi$ and every pair of distinct vertices $u,v \in V(G_k^g)$, there can be at most $\min\{ \deg_{G_k^g}(u),\deg_{G_k^g}(v) \}$ pairwise internally anticomplete paths between $u$ and $v$ in $G_k^g$. In particular, if $u$ and $v$ are not both in $B(G_k^g)$, there can be no more than three pairwise internally anticomplete paths between the two. Thus, we need only consider paths between two big vertices in $G_k^g$. We recall the notation used in defining $G_k^g$. Suppose $b_1 = P_{\ell_1}^{x_1} \in B(G_k^g)$ and $b_2 = P_{\ell_2}^{x_2} \in B(G_k^g)$ are two distinct vertices of $G_k^g$; by the symmetry of $G_k^g$, we may assume that $\ell_1 \leq \ell_2$ and $x_1 < x_2$. If $R$ is a path from $b_1$ to $b_2$ in $G_k^g$, we say that \emph{$R$ switches layers at\/ $x$, from\/ $\ell$ to\/ $\ell'$}, if $E(R)$ contains an edge with ends $P_\ell^x$ and $P_{\ell'}^x$. We emphasize that if a path switches layers at $x$ then it contains a big vertex with index $x$.

Suppose $\mathcal{R}$ is a set of paths from $b_1$ to $b_2$ that are pairwise internally anticomplete. Observe that, for every $R \in \mathcal{R}$, $R^*$ does not contain any big vertex with index $x_1$ or $x_2$, as $b_1$ and $b_2$ are the unique big vertices in $G_k^g$ with indices $x_1$ and $x_2$, respectively; this fact will be used in the proofs of the following two lemmas. For $R \in \mathcal{R}$, let $R^-$ denote the unique neighbor of $b_1$ in $R$, and let $R^+$ denote the unique neighbor of $b_2$ in $R$. We say that a path $R \in \mathcal{R}$ is \emph{standard} if $R^- = P_{\ell_1'}^{x_1}$ for some $\ell_1' > \ell_1$ and $R^+ = P_{\ell_2'}^{x_2}$ for some $\ell_2' > \ell_2$. Otherwise, we say that $R$ is \emph{nonstandard} (see Figure \ref{fig:nonstandard-path}). Note that $\mathcal{R}$ can include at most four nonstandard paths, since every nonstandard path uses either an edge of $P_{\ell_1}$ incident with $P_{\ell_1}^{x_1}$ or an edge of $P_{\ell_2}$ incident with $P_{\ell_2}^{x_2}$, and there are only four such edges. In what follows, we show that there may be at most three standard paths in $\mathcal{R}$.

Let $\mathcal{R}' = \{R^*\colon R$ is a standard path in $\mathcal{R}\}$. For $R^* \in \mathcal{R}'$, say that $R$ is an \emph{overpass} if, for some $\ell$ with $1 \leq \ell \leq k$, we have $V(P_\ell[x_1:x_2]) \subseteq V(R^*)$ and $P_\ell[x_1:x_2] \cap B(G_k^g) = \varnothing$ (see Figure \ref{fig:overpass}). For ease of notation, we will also say that $R^*$ is (or is not) an overpass to mean that $R$ is (or is not) an overpass.

\begin{figure}[t!]
\centering

\begin{tikzpicture}[scale=1,auto=left, line width=0.28mm]
\tikzstyle{every node}=[inner sep=1.2pt, fill=black,circle,draw]
\centering

    \draw (1, 1) node[draw=none,fill=none,left] {$P_5^0$} -- (12, 1) node[draw=none,fill=none,right] {$P_5^{64}$};
    \draw (1, 2) node[draw=none,fill=none,left] {$P_4^0$} -- (12, 2) node[draw=none,fill=none,right] {$P_4^{64}$};
    \draw (1, 3) node[draw=none,fill=none,left] {$P_3^0$} -- (12, 3) node[draw=none,fill=none,right] {$P_3^{64}$};
    \draw (1, 4) node[draw=none,fill=none,left] {$P_2^0$} -- (12, 4) node[draw=none,fill=none,right] {$P_2^{64}$};
    \draw (1, 5) node[draw=none,fill=none,left] {$P_1^0$} -- (12, 5) node[draw=none,fill=none,right] {$P_1^{64}$};

    \draw[color=red] (3.75,3) -- (7.875, 3);

    \foreach \i in {1,...,5} {
        \node() at (6.5, \i) {};
        \node() at (1,\i) {};
        \node at (12,\i) {};
    }

    \draw (6.5, 5) edge[line width=0.28mm, line width=0.28mm, out = -90, in = 90] (6.5, 4);
    \draw (6.5, 5) edge[line width=0.28mm, line width=0.28mm, out = -105, in = 105] (6.5, 3);
    \draw (6.5, 5) edge[line width=0.28mm, out = -112, in = 112] (6.5, 2);
    \draw (6.5, 5) edge[line width=0.28mm, out = -120, in = 120] (6.5, 1);
    \node[draw=red,fill=red] at (6.5, 3) {};

    \foreach \i in {1,...,4} {
        \node() at (3.75,\i) {};
        \node() at (9.25,\i) {};
    }
    
    \node[draw=red,fill=red] at (3.75,3) {};
    \draw (3.75, 4) edge[line width=0.28mm, out = -90, in = 90,draw=red] (3.75, 3);
    
    \draw (3.75, 4) edge[line width=0.28mm, out = -105, in = 105] (3.75, 2);
    \draw (3.75, 4) edge[line width=0.28mm, out = -112, in = 112] (3.75, 1);
    \draw (9.25, 4) edge[line width=0.28mm, out = -90, in = 90] (9.25, 3);
    \draw (9.25, 4) edge[line width=0.28mm, out = -105, in = 105] (9.25, 2);
    \draw (9.25, 4) edge[line width=0.28mm, out = -112, in = 112] (9.25, 1);

    \node[draw=red,fill=red,scale=2] at (3.75,4) {};

    \foreach \i in {1,...,3} {
        \node() at (2.375, \i) {};
        \node() at (5.125,\i) {};
        \node() at (7.875,\i) {};
        \node() at (10.625,\i) {};
    }
    
    \draw (2.375, 3) edge[line width=0.28mm, out = -90, in = 90] (2.375, 2);
    \draw (2.375, 3) edge[line width=0.28mm, out = -112, in = 112] (2.375, 1);
    \draw (5.125, 3) edge[line width=0.28mm, out = -90, in = 90] (5.125, 2);
    \draw (5.125, 3) edge[line width=0.28mm, out = -112, in = 112] (5.125, 1);
    \draw (7.875, 3) edge[line width=0.28mm, out = -90, in = 90] (7.875, 2);
    \draw (7.875, 3) edge[line width=0.28mm, out = -112, in = 112] (7.875, 1);
    \draw (10.625, 3) edge[line width=0.28mm, out = -90, in = 90] (10.625, 2);
    \draw (10.625, 3) edge[line width=0.28mm, out = -112, in = 112] (10.625, 1);

    \node[draw=red,fill=red] at (5.125,3) {};

    \node[draw=red,fill=red,scale=2] at (7.875,3) {};

    \foreach \i in {1,2} {
        \node() at (1.6875,\i) {};
        \node() at (3.0625,\i) {};
        \node() at (4.4375,\i) {};
        \node() at (5.8125,\i) {};
        \node() at (7.1875,\i) {};
        \node() at (8.5625,\i) {};
        \node() at (9.9375,\i) {};
        \node() at (11.3125,\i) {};
    }

    \draw(1.6875,2) edge[line width=0.28mm, out = -90, in = 90] (1.6875,1);
    \draw(3.0625,2) edge[line width=0.28mm, out = -90, in = 90] (3.0625,1);
    \draw(4.4375,2) edge[line width=0.28mm, out = -90, in = 90] (4.4375,1);
    \draw(5.8125,2) edge[line width=0.28mm, out = -90, in = 90] (5.8125,1);
    \draw(7.1875,2) edge[line width=0.28mm, out = -90, in = 90] (7.1875,1);
    \draw(8.5625,2) edge[line width=0.28mm, out = -90, in = 90] (8.5625,1);
    \draw(9.9375,2) edge[line width=0.28mm, out = -90, in = 90] (9.9375,1);
    \draw(11.3125,2) edge[line width=0.28mm, out = -90, in = 90] (11.3125,1);

\end{tikzpicture}

\caption{A nonstandard path between two (enlarged) big vertices}
\label{fig:nonstandard-path}
\end{figure}

\begin{figure}[t!]
\centering

\begin{tikzpicture}[scale=1,auto=left, line width=0.28mm]
\tikzstyle{every node}=[inner sep=1.2pt, fill=black,circle,draw]
\centering

    \draw (1, 1) node[draw=none,fill=none,left] {$P_5^0$} -- (12, 1) node[draw=none,fill=none,right] {$P_5^{64}$};
    \draw (1, 2) node[draw=none,fill=none,left] {$P_4^0$} -- (12, 2) node[draw=none,fill=none,right] {$P_4^{64}$};
    \draw (1, 3) node[draw=none,fill=none,left] {$P_3^0$} -- (12, 3) node[draw=none,fill=none,right] {$P_3^{64}$};
    \draw (1, 4) node[draw=none,fill=none,left] {$P_2^0$} -- (12, 4) node[draw=none,fill=none,right] {$P_2^{64}$};
    \draw (1, 5) node[draw=none,fill=none,left] {$P_1^0$} -- (12, 5) node[draw=none,fill=none,right] {$P_1^{64}$};

    \draw[color=red] (3.75,4) -- (9.25,4);

    \foreach \i in {1,...,5} {
        \node() at (6.5, \i) {};
        \node() at (1,\i) {};
        \node at (12,\i) {};
    }

    \draw (6.5, 5) edge[line width=0.28mm, line width=0.28mm, out = -90, in = 90] (6.5, 4);
    \draw (6.5, 5) edge[line width=0.28mm, line width=0.28mm, out = -105, in = 105] (6.5, 3);
    \draw (6.5, 5) edge[line width=0.28mm, out = -112, in = 112] (6.5, 2);
    \draw (6.5, 5) edge[line width=0.28mm, out = -120, in = 120] (6.5, 1);

    \node[color=red] at (6.5,4) {};

    \foreach \i in {1,...,4} {
        \node() at (3.75,\i) {};
        \node() at (9.25,\i) {};
    }
    
    \draw (3.75, 4) edge[line width=0.28mm, out = -90, in = 90] (3.75, 3);
    \draw (3.75, 4) edge[line width=0.28mm, out = -105, in = 105] (3.75, 2);
    \draw (3.75, 4) edge[line width=0.28mm, out = -112, in = 112,color=red] (3.75, 1);
    \draw (9.25, 4) edge[line width=0.28mm, out = -90, in = 90] (9.25, 3);
    \draw (9.25, 4) edge[line width=0.28mm, out = -105, in = 105] (9.25, 2);
    \draw (9.25, 4) edge[line width=0.28mm, out = -112, in = 112,color=red] (9.25, 1);

    \foreach \i in {1,...,3} {
        \node() at (2.375, \i) {};
        \node() at (5.125,\i) {};
        \node() at (7.875,\i) {};
        \node() at (10.625,\i) {};
    }
    
    \draw (2.375, 3) edge[line width=0.28mm, out = -90, in = 90] (2.375, 2);
    \draw (2.375, 3) edge[line width=0.28mm, out = -112, in = 112] (2.375, 1);
    \draw (5.125, 3) edge[line width=0.28mm, out = -90, in = 90] (5.125, 2);
    \draw (5.125, 3) edge[line width=0.28mm, out = -112, in = 112] (5.125, 1);
    \draw (7.875, 3) edge[line width=0.28mm, out = -90, in = 90] (7.875, 2);
    \draw (7.875, 3) edge[line width=0.28mm, out = -112, in = 112] (7.875, 1);
    \draw (10.625, 3) edge[line width=0.28mm, out = -90, in = 90] (10.625, 2);
    \draw (10.625, 3) edge[line width=0.28mm, out = -112, in = 112] (10.625, 1);

    \foreach \i in {1,2} {
        \node() at (1.6875,\i) {};
        \node() at (3.0625,\i) {};
        \node() at (4.4375,\i) {};
        \node() at (5.8125,\i) {};
        \node() at (7.1875,\i) {};
        \node() at (8.5625,\i) {};
        \node() at (9.9375,\i) {};
        \node() at (11.3125,\i) {};
    }

    \draw[color=red] (9.25,1) -- (7.1875,1) {};

    \draw(1.6875,2) edge[line width=0.28mm, out = -90, in = 90] (1.6875,1);
    \draw(3.0625,2) edge[line width=0.28mm, out = -90, in = 90] (3.0625,1);
    \draw(4.4375,2) edge[line width=0.28mm, out = -90, in = 90] (4.4375,1);
    \draw(5.8125,2) edge[line width=0.28mm, out = -90, in = 90,color=red] (5.8125,1);
    \draw(7.1875,2) edge[line width=0.28mm, out = -90, in = 90,color=red] (7.1875,1);
    \draw(8.5625,2) edge[line width=0.28mm, out = -90, in = 90] (8.5625,1);
    \draw(9.9375,2) edge[line width=0.28mm, out = -90, in = 90] (9.9375,1);
    \draw(11.3125,2) edge[line width=0.28mm, out = -90, in = 90] (11.3125,1);

    \node[color=red] at (3.75,1) {};
    \node[color=red] at (3.75,4) {};

    \node[color=red] at (9.25,1) {};
    \node[color=red] at (9.25,4) {};

    \node[color=red] at (5.125,1) {};
    \node[color=red] at (7.875,1) {};

    \node[color=red] at (4.4375,1) {};
    \node[color=red] at (8.5625,1) {};

    \node[color=red] at (5.8125,1) {};
    \node[color=red,scale=2] at (5.8125,2) {};

    \draw[color=red] (3.75,1) -- (5.8125,1);

    \node[color=red] at (7.1875,1) {};
    \node[color=red,scale=2] at (7.1875,2) {};

\end{tikzpicture}

\caption{An overpass between two (enlarged) big vertices}
\label{fig:overpass}
\end{figure}

\begin{lemma} \label{non-overpass}
    There is at most one\/ $R^* \in \mathcal{R}'$ that is not an overpass.
\end{lemma}

\begin{proof}
    For $R^* \in \mathcal{R}'$, say $v \in V(R^*)$ is an \emph{internal big vertex of\/ $R^*$} if $v \in B(G_k^g)$ and $v$ has index greater than $x_1$ and smaller than $x_2$.

    We first show that if $R^* \in \mathcal{R}'$ is not an overpass, then $R^*$ contains some internal big vertex $b = P_\ell^x$. Indeed, for each $x'$ with $x_1 \leq x' \leq x_2$, the path $R^*$ includes a vertex $P_{\ell'}^{x'}$ for some $\ell'$ with $1 \leq \ell' \leq k$. If, for some $\ell'$, $R^*$ contains all of $P_{\ell'}[x_1:x_2]$, then one of the vertices of $P_{\ell'}[x_1:x_2]$ is a big vertex as otherwise $R^*$ would be an overpass. On the other hand, if there is no $\ell'$ such that $R^*$ contains $P_{\ell'}^{x'}$ for all $x'$ with $x_1 \leq x' \leq x_2$, then $R^*$ switches layers at some index $x$ that is strictly between $x_1$ and $x_2$, and so $R^*$ contains a big vertex with index~$x$.

    Now assume that there is at least one element of $\mathcal{R}'$ that is not an overpass; let $b = P_{\ell_0}^{x_0}$ be a big vertex contained in some non-overpass of $\mathcal{R}'$ such that $x_1 < x_0 < x_2$ and $\ell_0$ is minimal. Let $R_0^*$ be the element of $\mathcal{R}'$ containing $b$, and suppose there is some other $R^* \in \mathcal{R}'$ that is not an overpass. The path $R^*$ contains some vertex with index $x_0$, say $v = P_{\ell}^{x_0}$. It follows that $\ell < \ell_0$, as otherwise $b$ and $v$ would be adjacent or equal. Note that $v$ is a small vertex as $b$ is a big vertex and $v$ has strictly smaller layer than $b$. Thus, there is some index $x'$ such that $R^*$ switches layers to $\ell$ at $x'$ and such that $P_\ell[x' : x_0]$ is contained in the subpath of $R^*$ from $R^-$ to $v$. It follows that $R^*$ includes a big vertex with layer at most $\ell$ and index $x'$. By the choice of $b$ and the fact that $\ell < \ell_0$, we deduce that either $x' < x_1$ or $x' > x_2$. The two cases are analogous; we show in detail how to handle the case where $x' < x_1$.

    We now show that $R^*$ contains $P_\ell[x_0:x_2]$ (and thus all of $P_\ell[x_1:x_2]$). Indeed, suppose this is not the case, and let $x''$ be maximal such that $P_\ell[x_0:x'']$ is contained in $R^*$, where $x_0 \leq x'' < x_2$. Since $R^*$ does not contain $P_\ell^{x'' + 1}$, $R^*$ switches layers at index $x''$ from $\ell$ to some other layer, and thus $R^*$ contains a big vertex with layer at most $\ell$ and index $x''$. But $\ell < \ell_0$ and $x_1 < x'' < x_2$, so this contradicts the choice of $b$.
\end{proof}

\begin{lemma} \label{leq-2-overpasses}
    $\mathcal{R}'$ contains at most two overpasses.
\end{lemma}

\begin{proof}
Let $R^* \in \mathcal{R}'$ be an overpass, and let $\ell \in \mathbb{N}$ with $1 \leq \ell \leq k$ be such that $V(P_\ell[x_1:x_2]) \subseteq V(R^*)$ and $P_\ell[x_1:x_2] \cap B(G_k^g) = \varnothing$.
    Note that, due to the construction of $G_k^g$, $P_\ell[x_1:x_2] \cap B(G_k^g) = \varnothing$ implies that $\ell \leq \ell_1$. Furthermore, as $R^*$ is the interior of a path from $b_1$ to $b_2$, it does not contain any edge incident to $b_1$, so it follows that $\ell \neq \ell_1$. Thus, for every overpass $R^* \in \mathcal{R}'$, there is some $\ell \in \mathbb{N}$ with $\ell < \ell_1$ such that $P_\ell^x \in V(R^*)$ for all $x$ with $x_1 \leq x \leq x_2$, and none of these vertices are big. It follows that $R^*$ contains a big vertex $b = P^{x_0}_{\ell'}$ for some $\ell' < \ell$, where $x_0 < x_1$.

    Now suppose that $R_\alpha^*$ and $R_\beta^*$ are two distinct elements of $\mathcal{R}'$. Let $b_\alpha = P_{\ell_\alpha}^{x_\alpha}$ be the big vertex in $R_\alpha^*$ with minimal distance to $R_\alpha^-$ in $R_\alpha^*$, subject to the condition $\ell_\alpha < \ell_1$. Similarly, let $b_\beta = P_{\ell_\beta}^{x_\beta}$ be the big vertex in $R_\beta^*$ with minimal distance to $R_\beta^-$ in $R_\beta^*$, subject to the condition $\ell_\beta < \ell_1$.

    We now show that either $x_\alpha < x_1 < x_\beta$ or $x_\beta < x_1 < x_\alpha$ holds. Suppose instead that both $x_\alpha < x_1$ and $x_\beta < x_1$ are true (the case where both are greater than $x_1$ is analogous). Without loss of generality, we assume that $x_\alpha > x_\beta$. Then the subpath of $R_\beta^*$ from $R_\beta^-$ to $b_\beta$ contains some vertex $v$ with index $x_\alpha$ and layer strictly larger than $\ell_\alpha$. But then $b_\alpha$ and $v$ are adjacent, which is a contradiction.

    Now suppose that $\mathcal{R}'$ contains three or more overpasses. Then there exist $R_\alpha^*, R_\beta^* \in \mathcal{R}'$ such that either $x_\alpha < x_1$ and $x_\beta < x_1$ or $x_\alpha > x_1$ and $x_\beta > x_1$, contrary to the claim of the previous paragraph. Thus, $\mathcal{R}'$ contains at most two overpasses.
\end{proof}

\begin{corollary} \label{leq-seven-paths}
    $|\mathcal{R}| \leq 7$.
\end{corollary}

\begin{proof}
    $\mathcal{R}$ contains at most four nonstandard paths. Of the standard paths in $\mathcal{R}$, there are at most two overpasses by Lemma \ref{leq-2-overpasses}, and at most one non-overpass by Lemma \ref{non-overpass}. Thus, there are at most seven paths in $\mathcal{R}$.
\end{proof}

\section{No outerstring induced subgraph of large treewidth}

In this section, we prove the following.

\begin{lemma} \label{outerstring-isg-bound}
    There exists\/ $L\in \poi$ such that for all\/ $g,k\in \poi$, every induced subgraph\/ $H$ of\/ $G_k^g$ that is an outerstring graph satisfies\/ $\tw(H) \leq L$.
\end{lemma}

We need several results from the literature. Let $G$ be a graph, and let $w\colon V(G) \to [0,1]$. For $X \subseteq V(G)$, we write $w(X)=\sum_{x \in X} w(x)$, and for a subgraph $H$ of $G$ (not necessarily induced), we write $w(H)$ for $\sum_{x \in V(H)} w(x)$. We say that $w$ is a \emph{weight function on\/ $G$} if $w(G) = 1$, and a \emph{weak weight function on\/ $G$} if $w(G) \leq 1$ (so all weight functions are weak weight functions). A set $X \subseteq V(G)$ is a \emph{$w$-balanced separator in\/ $G$} if $w(D) \leq \frac{1}{2}$ for every component $D$ of $G \setminus X$. Treewidth and balanced separators are closely related through the following lemmas.

\begin{lemma}[\cite{wallpaper, induced-subgraphs-5,params-tied-to-tw, RS-GMII}]
\label{sep-to-tw}
    Let\/ $m\in \poi$, and let\/ $G$ be a graph such that for every weight function\/ $w$ on\/ $G$, there is a\/ $w$-balanced separator\/ $X_w$ in\/ $G$ with\/ $|X_w| \leq m$. Then\/ $\tw(G) \leq 2m$.
\end{lemma}

\begin{lemma}[\cite{induced-subgraphs-5, cygan, RS-GMII}]
\label{tw-to-sep}
    For every graph\/ $G$ and every weak weight function\/ $w$ on\/ $G$, there is a\/ $w$-balanced separator in\/ $G$ of size at most\/ $\tw(G)+1$.
\end{lemma}

We also need the following.

\begin{theorem}[Korhonen \cite{grid-induced-minor-theorem}] \label{bound-no-wall-isg}
    For all\/ $d,t\in \poi$, there exists\/ $L\in \poi$ such that if\/ $G$ is a graph of maximum degree at most\/ $d$ that does not contain any subdivision of\/ $W_{t \times t}$ or the line graph of any subdivision of\/ $W_{t \times t}$ as an induced subgraph, then\/ $\tw(G)\leq L$.
\end{theorem}

A \emph{theta} is a graph $T$ consisting of two non-adjacent vertices $a$ and $b$ and three internally anticomplete paths $P_1$, $P_2$, and $P_3$ from $a$ to $b$, each of length at least $2$, and no other vertices or edges. We call $a$ and $b$ the \emph{ends of\/ $T$}, and the \emph{length of\/ $T$} is $\dist_T(a,b)$. We say that $T$ is an \emph{$\ell$-long theta} if its length is at least $\ell$.

Next, we recall a result implicit in \cite{polygon-visibility-graphs}, that $\ell$-long thetas are not outerstring graphs for $\ell \geq 4$. Since the class of outerstring graphs is hereditary, it follows that every graph containing an
$\ell$-long theta for $\ell \geq 4$ as an induced subgraph is not an outerstring graph either. This will be the main tool in the proof that every induced subgraph of our construction either has small treewidth or is not an outerstring graph.

Let $G$ be a graph, and let $\prec$ be a linear order on $V(G)$. For $X \subseteq V(G)$, we let $\prec_X$ denote the restriction of $\prec$ to the set $X$. We say that the outerstring representation of $G$ is \emph{$\prec$-constrained} if for all $u,v \in V(G)$, we have $u \prec v$ if and only if the point at which the curve corresponding to $u$ intersects the $x$-axis is to the left of the point at which the curve corresponding to $v$ intersects the $x$-axis. It follows that, for every $X\subseteq V(G)$, the set of all curves in the representation corresponding to the vertices in $X$ forms a $\prec_X$-constrained outerstring representation of $G[X]$. In particular, we have the following.

\begin{lemma} \label{non-outerstring-criteria}
    Let\/ $G$ be a graph. Assume that for every linear order\/ $\prec$ on\/ $V(G)$, there exists\/ $X \subseteq V(G)$ such that\/ $G[X]$ admits no\/ $\prec_X$-constrained outerstring representation. Then\/ $G$ is not an outerstring graph.
\end{lemma}

The following is implicit in \cite{polygon-visibility-graphs}.

\begin{lemma} \label{long-theta}
    For all\/ $\ell \geq 4$, $\ell$-long thetas are not outerstring graphs.
\end{lemma}

\begin{proof}
    Let $T$ be an $\ell$-long theta for some $\ell \geq 4$. In Proposition 6.2 of \cite{polygon-visibility-graphs}, it is shown that for every linear order $\prec$ on $V(T)$, there exists a $4$-subset $X = \{x_1,x_2,x_3,x_4\}$ of vertices such that $x_1 \prec x_2 \prec x_3 \prec x_4$ and $E(T[X]) = \{x_1x_3,x_2x_4\}$. Clearly, this means there is no $\prec_X$-constrained outerstring representation of $T[X]$. Hence, by Lemma~\ref{non-outerstring-criteria}, the graph $T$ is not an outerstring graph.
\end{proof}

In view of Lemmas \ref{sep-to-tw} and \ref{long-theta}, in order to prove Lemma \ref{outerstring-isg-bound}, it suffices to show that, for every induced subgraph $H$ of $G^g_k$ with no induced long theta and every weight function $w$, there is a small $w$-balanced separator in $H$. We do this by finding small balanced separators with respect to certain weight functions on certain induced minors of $H$, which can then be translated back into a small $w$-balanced separator in $H$.

To this end, we fix $g,k\in \poi$ and put $G = G_k^g$, $B = B(G)$, $M = M(G)$, and $S = S(G)$. Let $H$ be an induced subgraph of $G$, and let $w$ be a weight function on $H$. Let $B_H = B \cap V(H)$, $M_H = M \cap V(H)$, and $S_H = S \cap V(H)$. The graph $H$ inherits from $G$ the properties (\ref{props-1}), (\ref{props-2}), (\ref{props-3}), and (\ref{props-4}) from Lemma \ref{props}, which are restated here.

\begin{lemma}\label{props-H}
    The following hold.
    \begin{enumerate}[{\rm (i)}]
        \item \label{props-H-1} $H$ is triangle-free.
        \item \label{props-H-2} If\/ $u,v \in V(H)$ are adjacent, then\/ $\deg_H(u) \leq 3$ or\/ $\deg_H(v) \leq 3$, and\/ $\{u,v\} \nsubseteq B_H$.
        \item \label{props-H-3} If\/ $u,v \in B_H$ with\/ $u \neq v$, then\/ $\dist_{G}(u,v) \geq 2^g$.
        \item \label{props-H-4} If\/ $u \in V(H)$ and\/ $\deg_{G}(u) = 3$, then there is at most one\/ $v \in N_H(u)$ such that\/ $\deg_H(v) \geq 3$.
    \end{enumerate}
\end{lemma}

For $b \in B$, let $N_H^M(b) = N_H(b) \cap M$ and $N_H^M[b] = N_H^M(b) \cup \{b\}$. Put $B_{H'} = \{N_H^M[b]\colon b \in B_H\}$, $M_{H'} = \{\{m\}\colon m \in M_H,\: m \notin \bigcup_{X \in B_{H'}}X\}$, and $S_{H'} = \{\{s\}\colon s \in S_H\}$. Let $H'$ be the model in $H$ induced by $B_{H'} \cup M_{H'} \cup S_{H'}$, and define $w'\colon V(H') \to [0,1]$ by $w'(S) = \sum_{v \in S} w(v)$ for $S \in V(H')$. It is easy to see that $w'$ is a weight function on $H'$. We note also that every element of $B_{H'}$ is uniquely identified by a vertex of $B_H$, and every element of $M_{H'}$ or $S_{H'}$ is uniquely identified by a vertex of $M_H$ or $S_H$, respectively.

Note that $H'$ here is defined analogously to Section \ref{no-large-wall}, and so by Lemma~\ref{H'-SP}, we have $\tw(H') \leq 2$. It follows from Lemma \ref{tw-to-sep} that there is a $w'$-balanced separator $K' \subseteq V(H')$ for $H'$ of size at most $3$. It is straightforward to see that taking $K = \bigcup_{X \in K'} X$ gives a $w$-balanced separator in $H$. However, there is no bound on the size of the sets $X$, as each big vertex of $G$ can have up to $k$ medium neighbors, and our desired bound needs to be independent of $k$.

To remove this dependence on $k$, we now define a new induced minor of $H$ related to $H'$ and $K'$. First, we partition $K'$ by defining the following sets:
\begin{align*}
    Y'_{B,>3} &= \{ b \in B_H\colon N_H^M[b] \in K',\: \deg_H(b) > 3 \}, \\
    Y'_{B,\leq 3} &= \{ b \in B_H\colon N_H^M[b] \in K',\: \deg_H(b) \leq 3 \}, \\
    K'_{B,>3} &= \{N_H^M[b]\colon b \in Y'_{B,>3}\}, \\
    K'_{B,\leq 3} &= \{N_H^M[b]\colon b \in Y'_{B,\leq 3}\}, \\
    K'_M &= K' \cap M_{H'}, \\
    K'_S &= K' \cap S_{H'}.
\end{align*}
Note that $K' = K'_{B,>3} \cup K'_{B,\leq 3} \cup K'_M \cup K'_S$ and all of these subsets are pairwise disjoint. The set $K'_{B,>3}$ comprises ``troublesome'' vertices of $K'$ in the sense that they are the vertices preventing $\bigcup_{S \in K'} S$ from having bounded size.

Let $\mathcal{N} = \bigcup_{b \in Y'_{B,>3}} \{\{n\}\colon n \in N_H^M(b)\}$ and $\mathcal{D} = \{\bigcup_{X \in V(D)} X\colon D$ is a component of $H' \setminus K' \}$. Let $H''$ be the model in $H$ induced by $\mathcal{N} \cup \mathcal{D}$; it is straightforward to see that $H''$ is a bipartite graph with bipartition $(\mathcal{N},\mathcal{D})$. Define $w''\colon V(H'') \to [0,1]$ by $w''(S) = \sum_{v \in S} w(v)$ for $S \in V(H'')$. Since $\bigcup_{S \in V(H'')} S \subseteq V(H)$, $w''$ is a weak weight function on $H''$. To bound the treewidth of $H''$, we need the following lemma.

\begin{lemma}[\cite{wallpaper, submodular-functions}] \label{three-vertices}
    Let\/ $T$ be a graph that does not contain\/ $K_3$ as a subgraph. Let\/ $x_1,x_2,x_3$ be distinct vertices of\/ $T$, and assume that\/ $F$ is a connected induced subgraph of\/ $T \setminus \{x_1,x_2,x_3\}$ such that\/ $V(F)$ contains at least one neighbor of each of\/ $x_1,x_2,x_3$, and that\/ $V(F)$ is minimal subject to inclusion. Then, one of the following holds.

    \begin{enumerate}[{\rm (i)}]
        \item \label{three-vertices-1} For some distinct\/ $i,j,k \in \{1,2,3\}$, there exists\/ $P$ that is either a path from\/ $x_i$ to\/ $x_j$ or a hole containing the edge\/ $x_ix_j$ such that
        \begin{itemize}
            \item $V(F) = V(P) \setminus \{x_i,x_j\}$, and
            \item $x_k$ has at least two non-adjacent neighbors in\/ $F$.
        \end{itemize}
        \item \label{three-vertices-2} There is a vertex\/ $a \in V(F)$ and three paths\/ $P_1,P_2,P_3$, where\/ $P_i$ is a path from\/ $a$ to\/ $x_i$, such that
        \begin{itemize}
            \item $V(F) = (V(P_1) \cup V(P_2) \cup V(P_3)) \setminus \{x_1,x_2,x_3\}$, and
            \item the sets\/ $V(P_1) \setminus \{a\}$, $V(P_2) \setminus \{a\}$, and\/ $V(P_3) \setminus \{a\}$ are pairwise disjoint, and
            \item for distinct\/ $i,j \in \{1,2,3\}$, there are no edges between\/ $V(P_i)$ and\/ $V(P_j)$ except possibly\/ $x_ix_j$.
        \end{itemize}
    \end{enumerate}

\end{lemma}

\begin{lemma} \label{H''-bounded-degree}
    If\/ $H''$ has maximum degree at least\/ $9$, then\/ $H$ contains an\/ $\ell$-long theta as an induced subgraph for some\/ $\ell \geq 2^g - 1$. In particular, if\/ $H$ is an outerstring graph, then\/ $H''$ has maximum degree less than\/ $9$.
\end{lemma}

\begin{proof}
    Suppose that $H''$ has maximum degree at least $9$. For $\{n\} \in \mathcal{N}$, we have $n \in M_H$ by construction, so $\deg_H(n) \leq 3$ and thus $\deg_{H''}(\{n\}) \leq 3$. This means that there is some $D \in \mathcal{D}$ with $\deg_{H''}(D) \geq 9$. Since $|K'_{B,>3}| \leq |K'| = 3$, there is $b \in Y'_{B,>3}$ such that there are at least three edges of $H''$ with one end $D$ and the other end in $\{\{n\}\colon n \in N_H^M(b)\}$. It follows that there exist distinct $n_1,n_2,n_3 \in N_H^M(b)$ such that $N_H(n_i) \cap D \neq \varnothing$ for every $i \in \{1,2,3\}$.

    Let $X \subseteq D$ be minimal (with respect to inclusion) such that $H[X]$ is connected and contains at least one neighbor of each of $n_1,n_2,n_3$. We now apply Lemma \ref{three-vertices} with $n_1,n_2,n_3$ and $H[X]$.

    Suppose that case (\ref{three-vertices-1}) applies; let $\{i,j,k\} = \{1,2,3\}$ and $P$ be such that $P$ is a path in $H$ from $n_i$ to $n_j$, $n_k$ has at least two non-adjacent neighbors in $P$, and $V(H[X]) = V(P) \setminus \{n_i, n_j\}$. Note that $P$ is not a hole, since $n_i$ and $n_j$ are both adjacent to $b$, thus they are not adjacent to each other. Since $n_k$ is adjacent to $b$ in $H$, $n_k$ is a medium vertex, thus $\deg_G(n_k) = 3$. Furthermore, every neighbor of $n_k$ in $P$ (of which there are at least $2$) has degree at least $3$ in $H$. This contradicts Lemma \ref{props-H} (\ref{props-H-4}).

    Suppose instead that case (\ref{three-vertices-2}) applies; let $a \in X$ and $P_1,P_2,P_3$ be such that $P_i$ is a path from $a$ to $n_i$ for each $i \in \{1,2,3\}$, and $X = (V(P_1) \cup V(P_2) \cup V(P_3)) \setminus \{n_1,n_2,n_3\}$, and the sets $V(P_1) \setminus \{a\}$, $V(P_2) \setminus \{a\}$, and $V(P_3) \setminus \{a\}$ are pairwise anticomplete. Note that $\dist_H(a, b) \geq 2^g-1$; this follows from Lemma \ref{props-H} \eqref{props-H-3} as $b \in B(G)$ and $a \in V(D) \subseteq G \setminus N[b]$ and $a$ has degree at least 3 in $G$. So either $a \in B(G)$ and \eqref{props-H-3} applies, or $a$ is adjacent to a vertex $a'$ in $B(G) \setminus \{b\}$ and the statement follows from \eqref{props-H-3} applied to $a'$ and $b$.

    We now have that $H[V(P_1) \cup V(P_2) \cup V(P_3) \cup \{a,b\}]$ is an induced subgraph of $H$ that is a theta of length at least $2^g-1$ (see Figure \ref{fig:theta-in-bipartite-graph}), completing the proof. Note that if $H$ is an outerstring graph, then such an induced subgraph contradicts Lemma \ref{long-theta}, implying that $H''$ has maximum degree less than $9$.
\end{proof}

\begin{figure}
    \centering
    \includegraphics[width=0.4\linewidth]{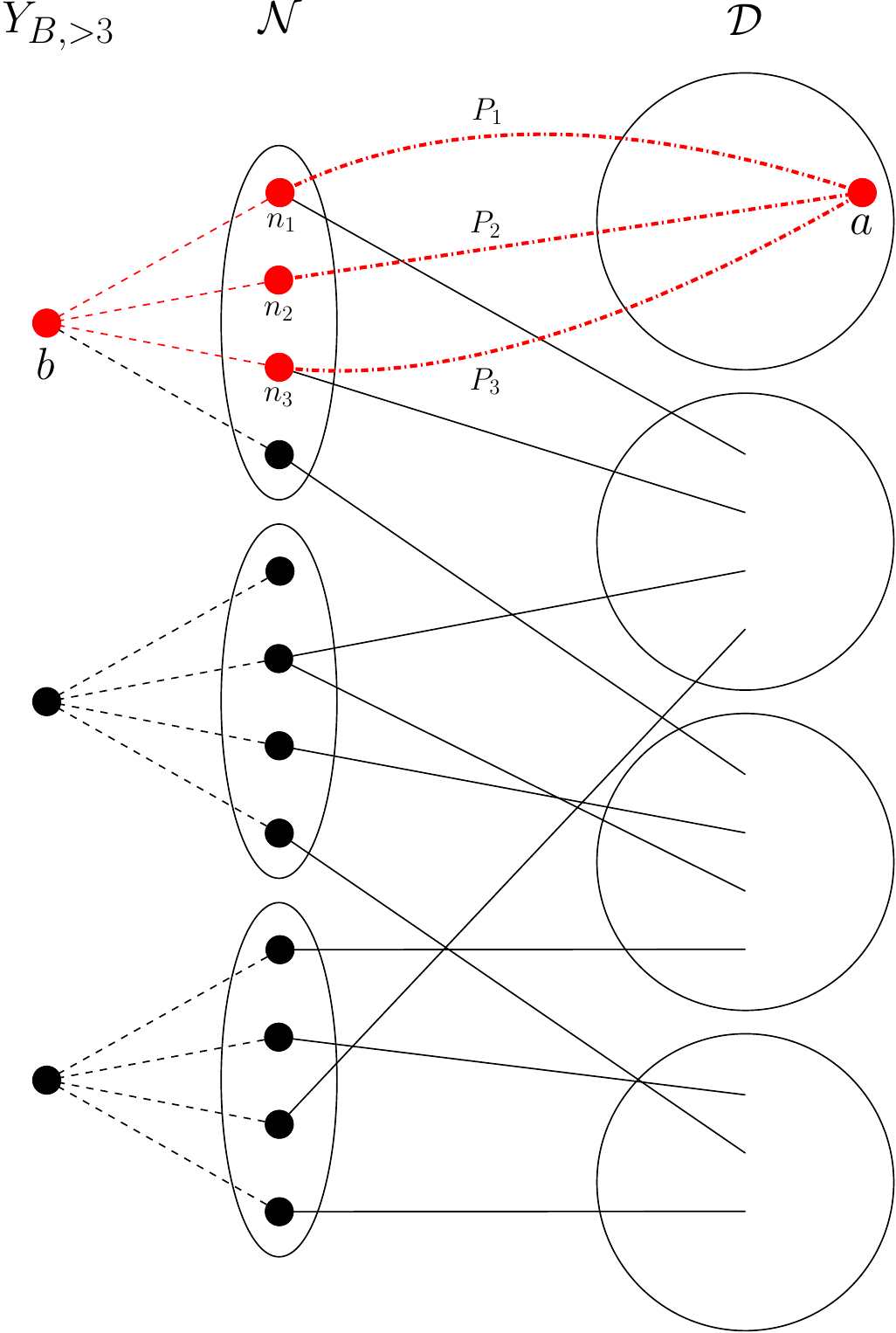}
    \caption{An illustration of the theta that appears in $H$ when $H''$ has maximum degree at least $9$.}
    \label{fig:theta-in-bipartite-graph}
\end{figure}

From here onwards, we assume that $H$ is an outerstring graph, so that we may proceed to bound the treewidth of $H''$ via Lemma \ref{H''-bounded-degree}.

\begin{corollary} \label{H''-tw}
    There exists\/ $L''\in \poi$ (independent of\/ $k$ and\/ $H$) such that\/ $\tw(H'') \leq L''$.
\end{corollary}

\begin{proof}
    Let $h_0$ be as in Lemma \ref{no-wall-im}. Since $H''$ is an induced minor of $G$ and $G$ does not contain $W_{h_0 \times h_0}$ as an induced minor, it follows that $H''$ does not contain $W_{h_0 \times h_0}$ as an induced minor, and so in particular $H''$ does not contain any subdivision of $W_{h_0 \times h_0}$ or $L(W_{h_0 \times h_0})$ as an induced subgraph. By Lemma \ref{H''-bounded-degree}, $H''$ has maximum degree less than $9$. The result now follows from Theorem \ref{bound-no-wall-isg}.
\end{proof}

By Lemma \ref{tw-to-sep}, $H''$ has a $w''$-balanced separator $K''$ of size at most $L'' + 1$. Let $K''_\mathcal{N} = K'' \cap \mathcal{N}$ and $K''_\mathcal{D} = K'' \cap \mathcal{D}$.

\begin{lemma} \label{H''-balanced-sep} Let\/ $L''$ be as in Corollary \ref{H''-tw}. Then
    \begin{enumerate}[{\rm (i)}]
        \item \label{H''-balanced-sep-1} $K''_\mathcal{N} \cup N_{H''}(K''_\mathcal{D})$ is a\/ $w''$-balanced separator in\/ $H''$, and
        \item \label{H''-balanced-sep-2} $|K''_\mathcal{N} \cup N_{H''}(K''_\mathcal{D})| \leq 9(L''+1)$.
    \end{enumerate}
\end{lemma}

\begin{proof}
    Let $C$ be a component of $H'' \setminus (K''_\mathcal{N} \cup N_{H''}(K''_\mathcal{D}))$. First suppose that $V(C) \cap K''_\mathcal{D} \neq \varnothing$, and let $D \in V(C) \cap K''_\mathcal{D}$. Since $N_{H''}(D) \subseteq K''_\mathcal{N} \cup N_{H''}(K''_\mathcal{D})$, we have $\deg_{H'' \setminus (K''_\mathcal{N} \cup N_{H''}(K''_\mathcal{D}))}(D) = 0$, and so $V(C) = \{D\}$. By the definitions of $\mathcal{D}$ and $w''$, it follows that $w''(D) \leq \frac{1}{2}$. Now suppose that $V(C) \cap K''_\mathcal{D} = \varnothing$. Then $C$ is a connected induced subgraph of $H'' \setminus K''$, so in particular, there exists a component $C^*$ of $H'' \setminus K''$ such that $V(C) \subseteq V(C^*)$. It follows that $w''(C) \leq w''(C^*) \leq \frac{1}{2}$. This proves (\ref{H''-balanced-sep-1}).

    To see that (\ref{H''-balanced-sep-2}) holds, we observe that $K''_\mathcal{N} \cup N_{H''}(K''_\mathcal{D})$ is obtained from $K''$ by removing a subset of its elements and replacing each removed element by at most nine new elements. Since $|K''| \leq L'' + 1$, the bound follows.
\end{proof}

We are ready to translate the balanced separators for $H'$ and $H''$ back into a $w$-balanced separator in $H$.

\begin{lemma} \label{H-balanced-sep}
    Let\/ $K^* = K'_{S} \cup K'_{M} \cup K'_{B,\leq 3} \cup K''_\mathcal{N} \cup N_{H''}(K''_\mathcal{D})$, and let\/ $K = (\bigcup_{X \in K^*} X) \cup Y'_{B,> 3}$. Then the following hold.
    \begin{enumerate}[{\rm (i)}]
        \item \label{H-balanced-sep-1} $K$ is a\/ $w$-balanced separator in\/ $H$.
        \item \label{H-balanced-sep-2} $|K| \leq 21 + 9(L'' + 1)$.
    \end{enumerate}
\end{lemma}

\begin{proof}
    The induced subgraph $H \setminus ((\bigcup_{X \in K'_S \cup K'_M \cup K'_{B,\leq 3}} X) \cup Y'_{B,> 3})$ of $H$, which will be denoted by $F$, has vertex set $\{n\colon \{n\} \in \mathcal{N}\} \cup (\bigcup_{D \in \mathcal{D}} D)$. Thus, for every component $C$ of $F \setminus (\{n\colon \{n\} \in K''_\mathcal{N} \cup N_{H''}(K''_\mathcal{D})\})$, there is a corresponding component $C''$ of $H'' \setminus (K''_\mathcal{N} \cup N_{H''}(K''_\mathcal{D}))$ such that $V(C) = \bigcup_{X \in V(C'')} X$, and so it follows that $w(C) = w''(C'')$. By Lemma \ref{H''-balanced-sep} (\ref{H''-balanced-sep-1}), we have $w''(C'') \leq \frac{1}{2}$, thus $w(C) \leq \frac{1}{2}$. This proves (\ref{H-balanced-sep-1}).

    Since $K'_S$, $K'_M$, and $K'_{B,\leq 3}$ are all subsets of $K'$, they each have size at most $3$. Every element of $K'_S$ and $K'_M$ is a singleton set, and so $|\bigcup_{\{s\} \in K'_S} \{s\}| \leq 3$ and $|\bigcup_{\{m\} \in K'_M} \{m\}| \leq 3$. Every $N_H^M[b] \in K'_{B,\leq 3}$ has size at most $4$, and so
    \[\biggl\lvert\bigcup_{N_H^M[b] \in K'_{B,\leq 3}} N_H^M[b]\biggr\rvert \leq 12.\]
    Since $|K''_\mathcal{N} \cup N_{H''}(K''_\mathcal{D})| \leq 9(L''+1)$ by Lemma \ref{H''-balanced-sep} (\ref{H''-balanced-sep-2}) and every element of $K''_\mathcal{N} \cup N_{H''}(K''_\mathcal{D})$ is a singleton, we have
    \[\biggl\lvert\bigcup_{\{n\} \in K''_\mathcal{N} \cup N_{H''}(K''_\mathcal{D})} \{n\}\biggr\rvert \leq 9(L'' + 1).\]
    Moreover, the set $Y'_{B,> 3}$ has the same size as $K'_{B,> 3}$, and thus it has size at most $3$ (as $K'_{B,>3} \subseteq K'$). Combining these bounds, we have $|K| \leq 3 + 3 + 12 + 9(L'' + 1) + 3 = 21 + 9(L'' + 1)$. This proves (\ref{H-balanced-sep-2}).
\end{proof}

We are now ready to prove Lemma \ref{outerstring-isg-bound}.

\begin{proof}[Proof of Lemma \ref{outerstring-isg-bound}.]
    Let $g,k \in \poi$, and let $H$ be an induced subgraph of $G_k^g$ that is an outerstring graph. Let $L''$ be as in Corollary \ref{H''-tw}. Then for every weight function $w$ on $H$, by Lemma \ref{H-balanced-sep}, there is a $w$-balanced separator $X_w$ with $|X_w| \leq 21 + 9(L'' + 1)$. The claim now follows from Lemma \ref{sep-to-tw}, with $L = 42 + 18(L'' + 1)$.
\end{proof}

\section{Completing the proof}

We are now ready to complete the proof of our main result. As discussed at the end of Section~\ref{sec:construction}, we only need to show the following.

\begin{theorem}
    There exist\/ $t_0,L\in \poi$ such that for all\/ $g,k \in \poi$, the following hold.
    \begin{enumerate}[{\rm (i)}]
        \item $G_k^g$ is\/ $W_{t_0 \times t_0}$-induced-minor-free and\/ $K_{t_0,t_0}$-induced-minor-free; and
        \item if\/ $H$ is an induced subgraph of\/ $G_k^g$ and\/ $H$ is an outerstring graph, then\/ $\tw(H) \leq L$.
    \end{enumerate}
\end{theorem}

\begin{proof}
    Let $h_0$ and $r_0$ be as in Lemmas \ref{no-wall-im} and \ref{no-bipartite-im}, and let $t_0 = \max \{h_0,r_0\}$. Let $L$ be as in Lemma \ref{outerstring-isg-bound}. Then, by Lemmas \ref{no-wall-im} and \ref{no-bipartite-im}, $G_k^g$ does not contain $W_{t_0 \times t_0}$ or $K_{t_0,t_0}$ as an induced minor, and by Lemma \ref{outerstring-isg-bound}, if an induced subgraph $H$ of $G_k^g$ is an outerstring graph, then $\tw(H) \leq L$.
\end{proof}

We remark that, due to the more general setup of Lemma \ref{H''-bounded-degree}, our proof in fact gives the following stronger statement.

\begin{theorem}
\label{stronger}
    There exist\/ $t_0,L\in \poi$ such that for all\/ $g,k \in \poi$, the following hold.
    \begin{enumerate}[{\rm (i)}]
        \item $G_k^g$ is\/ $W_{t_0 \times t_0}$-induced-minor-free and\/ $K_{t_0,t_0}$-induced-minor-free; and
        \item if\/ $H$ is an induced subgraph of\/ $G_k^g$ and\/ $\tw(H) > L$, then\/ $H$ contains an\/ $\ell$-long theta as an induced subgraph for some\/ $\ell \geq 2^g-1$.
    \end{enumerate}
\end{theorem}

\section*{Acknowledgements}

This work was partly done during the 2024 Barbados Graph Theory Workshop, and the 2025 Oberwolfach Workshop on Graph Theory. We thank the organizers for inviting us and for creating a stimulating work environment. We also thank Rose McCarty and Nicolas Trotignon for many helpful discussions.

\bibliographystyle{plain}

\end{document}